\documentclass[11pt]{amsart}

\usepackage[T1]{fontenc}
\usepackage[utf8]{inputenc}
\usepackage[english]{babel}
\usepackage{relsize}
\usepackage{verbatim}
\usepackage{mathtools,amssymb,amsthm}

    \usepackage[vertfit]{breakurl}

\usepackage[maxnames=50,backend=biber,giveninits=true]{biblatex}
\usepackage{csquotes}
\usepackage[shortlabels]{enumitem}
\usepackage{multirow}
\setlist{nolistsep}
\setenumerate{label={{\normalfont({\roman*})}},leftmargin=6mm,itemsep=3pt,topsep=3pt}
\setitemize{label={$\vcenter{\hbox{\tiny$\bullet$}}$},leftmargin=6mm,itemsep=3pt,topsep=3pt}
\usepackage{caption,subcaption}

\usepackage[margin=3cm,includefoot,footskip=30pt]{geometry}

\usepackage{array}
\newcolumntype{C}[1]{>{\arraybackslash$}p{#1}<{$}}

\usepackage{titlesec}
\titleformat{\section}{\normalfont\rmfamily\Large\bfseries}{\thesection.}{14pt}{}
\titlespacing{\section}{0pt}{14pt}{7pt}

\titleformat{\subsection}{\normalfont\rmfamily\large\bfseries}{\thesubsection.}{12pt}{}
\titlespacing{\subsection}{0pt}{12pt}{6pt}

\renewcommand{\emptyset}{\varnothing}

\renewcommand{\leq}{\leqslant}

\renewcommand{\geq}{\geqslant}

\usepackage[algoruled,linesnumbered,vlined]{algorithm2e}
\usepackage{algpseudocode}

\usepackage{ifthen}
\usepackage{tikz}
\usetikzlibrary{matrix,fit,calc,positioning}
\tikzstyle{every picture}=[line join=round,line cap=round,line width=.75pt,every label/.append style={font=\small},label distance=-1.5pt]
\tikzstyle{every node}=[font=\small]
\tikzset{>=stealth}

\newcommand{\Z}{\mathbb Z}

\newcommand{\R}{\mathbb R}

\newcommand{\aff}{\mathrm{aff}}
\newcommand{\conv}{\mathrm{conv}}
\newcommand{\col}{\mathrm{col}}

\newcommand{\pr}{\mathrm{Pr}}

\newcommand{\rk}{\mbox{rk}}

\newcommand{\calB}{\mathcal{B}}

\newtheorem{prop}{Proposition}
\newtheorem{lem}[prop]{Lemma}

\newtheorem{conj}[prop]{Conjecture}
\newtheorem{thm}[prop]{Theorem}
\newtheorem{cor}[prop]{Corollary}

\theoremstyle{definition}

\newcommand{\restr}[1]{\big|_{#1}}

\newcommand{\op}{product}

\makeatletter
\renewcommand\paragraph{\@startsection{paragraph}{4}{\z@}%
{3.25ex \@plus1ex \@minus.2ex}{-.2em}%
{\normalfont\normalsize\bfseries}}
\makeatother

\begin{document}

\title{Recognizing Cartesian products\\ of matrices and polytopes}
\author{Manuel Aprile, Michele Conforti, Yuri Faenza,\\ Samuel Fiorini, Tony Huynh, Marco Macchia}
\date{}
%
%
\maketitle


\begin{abstract}
    The \emph{$1$-\op} of matrices $S_1 \in \R^{m_1 \times n_1}$ and $S_2 \in \R^{m_2 \times n_2}$ is the matrix in $\R^{(m_1+m_2) \times (n_1n_2)}$
whose columns are the concatenation of each column of $S_1$ with each column of $S_2$.  Our main result is a polynomial time algorithm for the following problem:  given a matrix $S$, is $S$ a $1$-\op, up to permutation of rows and columns?  
Our main motivation is a close link between the 1-\op{} of matrices and the Cartesian product of polytopes, which goes through the concept of slack matrix. Determining whether a given matrix is a slack matrix is an intriguing problem whose complexity is unknown, and our algorithm reduces the problem to irreducible instances. 
Our algorithm is based on minimizing a symmetric submodular function that expresses mutual information in information theory.  
We also give a polynomial time algorithm to recognize a more complicated matrix product, called the \emph{$2$-\op}.  Finally, as a corollary of our $1$-\op{} and $2$-\op{} recognition algorithms, we obtain a polynomial time algorithm to recognize slack matrices of $2$-level matroid base polytopes.
\end{abstract}

\section{Introduction}


Determining if an object can be decomposed as the `product' of two simpler objects is a ubiquitous theme in mathematics and computer science.  For example, every integer $n \geq 2$ has a unique factorization into primes, and every finite abelian group is the direct sum of cyclic groups.  Moreover, algorithms to efficiently \emph{find} such `factorizations' are widely studied, since many algorithmic problems are easy on indecomposable instances. In this paper, our objects of interest are matrices and polytopes.

For a matrix $S$, we let $S^\ell$ be the $\ell$th column of $S$.
The \emph{$1$-\op} of $S_1 \in \R^{m_1 \times n_1}$ and $S_2 \in \R^{m_2 \times n_2}$ is the matrix $S_1 \otimes S_2 \in \R^{(m_1+m_2) \times (n_1n_2)}$ such that for each $j \in  [n_1\cdot n_2]$, 
$$
(S_1 \otimes S_2)^{j} := \begin{pmatrix}
S_1^{k} \\ S_2^{\ell}
\end{pmatrix},
$$
where $k \in [n_1]$ and $\ell \in [n_2]$ satisfy $j=(k-1)n_2+\ell$. For example, 
$$
\begin{pmatrix} 1 &0 \end{pmatrix} \otimes \begin{pmatrix} 0 \end{pmatrix} =
\begin{pmatrix}
1 & 0\\
0 & 0
\end{pmatrix}
, \qquad
\begin{pmatrix} 1 &0 
\\ 2 & 3\end{pmatrix} \otimes \begin{pmatrix} 1 &0&0 \\ 0  & 1 & 1 \\ \end{pmatrix}=\begin{pmatrix} 1 &1 &1 &0&0 &0 
\\ 2 & 2& 2& 3& 3& 3\\ 1 &0&0 & 1&0&0\\ 0  & 1 & 1 & 0  & 1 & 1 \end{pmatrix}.
$$
Two matrices are \emph{isomorphic} if one can be obtained from the other by permuting rows and columns. A matrix $S$ is a \emph{1-\op} if there exist two non-empty matrices $S_1$  and $S_2$ such that $S$ is isomorphic to $S_1\otimes S_2$.  The following is our first main result.

\begin{thm} \label{thm:main1}
Given $S\in \R^{m\times n}$, there is an algorithm that is polynomial in $n,m$ 
which correctly determines if $S$ is a 1-\op{} and, in case it is, outputs two matrices $S_1, S_2$ such that $S_1\otimes S_2$ is isomorphic to $S$.
\end{thm}

A straightforward implementation of our algorithm would run in $O(m^3(m+n))$ time. However, below we do not explicitly state the running times of our algorithms nor try to optimize them.

The proof of Theorem~\ref{thm:main1} is by reduction to symmetric submodular function minimization using the concept of \emph{mutual information} from information theory.  Somewhat surprisingly, we do not know of a simpler proof of Theorem~\ref{thm:main1}.  

Our main motivation for Theorem~\ref{thm:main1} is geometric.  If $P_1 \subseteq \R^{d_1}$ and $P_2 \subseteq \R^{d_2}$ are polytopes, then their \emph{Cartesian product} is the polytope $P_1 \times P_2 := \{(x_1,x_2) \in \R^{d_1} \times \R^{d_2} \mid x_1 \in P_1,\ x_2 \in P_2\}$. 

Notice that if $P$ is given by an irredundant inequality description, determining if $P=P_1 \times P_2$ for some polytopes $P_1, P_2$ amounts to determining whether the constraint matrix can be put in block diagonal structure.  If $P$ is given as a list of vertices, then the algorithm of Theorem~\ref{thm:main1} determines if  $P$ is a Cartesian product.

Furthermore it turns out that the 1-\op{} of matrices corresponds to the Cartesian product of polytopes if we represent a polytope via its slack matrix, which we now describe.  
 
Let $P = \conv(\{v_1,\dots, v_n\}) = \{x\in \mathbb{R}^d \mid Ax \leq b\}$, where $\{v_1,\dots,v_n\} \subseteq \R^d$, $A \in \R^{m \times d}$ and $b \in \R^m$. The \emph{slack matrix} associated to these descriptions of $P$ is the matrix $S \in \R^{m \times n}_+$ with $S_{i,j} := b_i-A_i v_j$.  That is, $S_{i,j}$ is the slack of point $v_j$ with respect to the inequality $A_i x \leq b_i$.
 
Slack matrices were introduced in a seminal paper of Yannakakis~\cite{yannakakis1991expressing}, as a tool for reasoning about the extension complexity of polytopes (see \cite{conforti2013extended}).

 
Our second main result is the following corollary to Theorem~\ref{thm:main1}.
 
\begin{thm} \label{thm:main2}
Given a polytope $P$ represented by its slack matrix $S \in \R^{m \times n}$, there is an algorithm that is polynomial in $m,n$ which correctly determines if $P$ is affinely equivalent to a Cartesian product $P_1 \times P_2$ and, in case it is, outputs two matrices $S_1, S_2$ such that $S_i$ is the slack matrix of $P_i$, for $i \in [2]$.
\end{thm}

Some comments are in order here. First, our algorithm determines whether a polytope $P$ is affinely
equivalent to a Cartesian product of two polytopes. As affine transformations do not preserve the property of being Cartesian product, this is a different problem than that of determining whether $P$ \emph{equals} $P_1\times P_2$ for some polytopes $P_1, P_2$.
%
%
%
Second, the definition of 1-\op{} can be extended to a more complex operation which we call 2-\op. Theorems~\ref{thm:main1} and \ref{thm:main2} can be extended to handle 2-\op{}s, see Theorem \ref{thm:2-sum_recog}.
 
Slack matrices are fascinating objects, that are not fully understood. For instance, given a matrix $S \in \R^{m \times n}_+$, the complexity of determining whether $S$ is the slack matrix of some polytope is open. In \cite{gouveia2013nonnegative}, the problem has been shown to be equivalent to the Polyhedral Verification Problem (see~\cite{kaibel2003some}): given a vertex description of a polytope $P$, and an inequality description of a polytope $Q$, determine whether $P=Q$.  
 
Polytopes that have a $0/1$-valued slack matrix are called \emph{$2$-level polytopes}. These form a rich class of polytopes including stable set polytopes of perfect graphs, Birkhoff, and Hanner polytopes (see~\cite{aprile2018thesis, aprile20182, macchia2018two} for more examples and details).  We conjecture that slack matrix recognition is polynomial for $2$-level polytopes.  
 
\begin{conj} \label{conj:2levelrec}
Given $S \in \{0,1\}^{m \times n}$, there is an algorithm that is polynomial in $m,n$ which correctly determines if $S$ is the slack matrix of a polytope.
\end{conj}
 
Conjecture~\ref{conj:2levelrec} seems hard to settle: however
it has been proven for certain restricted classes of $2$-level polytopes, most notably for stable set polytopes of perfect graphs \cite{aprile2018thesis}.  As a final result, we apply Theorem \ref{thm:main1} and its extension to 2-\op{}s to 
show that Conjecture~\ref{conj:2levelrec} holds for $2$-level matroid base polytopes (precise definitions will be given later).
 
\begin{thm} \label{thm:main3}
Given $S \in \{0,1\}^{m \times n}$, there is an algorithm that is polynomial in $m,n$ which correctly determines if $S$ is the slack matrix of a $2$-level matroid base polytope.
\end{thm}

 {\bfseries{Paper Outline.}}
In Section \ref{sec:def} we study the properties of 1-\op{}s and 2-\op{}s in terms of slack matrices, proving Lemmas \ref{lem:1-sum_slack}, \ref{lem:2-sum_slack}. In Section \ref{sec:algorithms} we give algorithms to efficiently recognize 1-\op{}s and 2-\op{}s (Theorems \ref{thm:main1}, \ref{thm:2-sum_recog}), as well as showing a unique decomposition result for 1-\op{}s (Lemma \ref{lem:unique_decomposition}).
Finally, in Section \ref{sec:slackmatroids} we apply the previous results to slack matrices of matroid base polytopes, obtaining Theorem \ref{thm:main3}.

The results presented in this paper are contained in the PhD thesis of the first author \cite{aprile2018thesis}, to which we refer for further details.

\section{Properties of 1-\op{}s and 2-\op{}s}\label{sec:def}

Here we study the $1$-\op{} of matrices defined above in the introduction, as well as the $2$-\op{}.
We remark that the notion of $2$-\op{} and the related results can be generalized to $k$-\op{}s for every $k \geq 3$ (see~\cite{aprile2018thesis} for more details). The $k$-\op{} operation  is similar to that of glued product of polytopes in~\cite{margot1995composition}, except that the latter is defined for 0/1 polytopes, while we deal with general matrices.

We show that, under certain assumptions, the operations of $1$- and $2$-\op{}  preserve the property of being a slack matrix.
We recall the following characterization of slack matrices, due to~\cite{gouveia2013nonnegative}. 

We will denote the set of column vectors of a matrix $S$ by $\col(S)$. 

\begin{thm}[Gouveia~\emph{et al.} \cite{gouveia2013nonnegative}]\label{thm:char_slack-matrices}
Let $S\in \R^{m \times n}$ be a nonnegative matrix of rank at least 2. Then $S$ is the slack matrix of a polytope if and only if $\conv(\col(S)) = \aff(\col(S)) \cap \R_+^m$. Moreover, if $S$ is the slack matrix of polytope $P$ then $P$ is affinely equivalent to $\conv(\col(S))$.
\end{thm}

Throughout the paper, we will assume that the matrices we deal with are of rank at least $2$, so to apply Theorem \ref{thm:char_slack-matrices} directly. 


We point out that the slack matrix of a polytope $P$ is not unique, as it depends on the given descriptions of $P$. We say that a slack matrix is \emph{non-redundant} if its rows bijectively correspond to the facets of $P$ and its columns bijectively correspond to the vertices of $P$. In particular, non-redundant slack matrices do not contain two identical rows or columns, nor rows or columns which are all zeros, or all non-zeros. They are unique up to permuting rows and columns, and scaling rows with positive reals.

\subsection{$1$-\op{}s}

We show that the 1-\op{} operation preserves the property of being a slack matrix. 

\begin{lem}\label{lem:1-sum_slack}
Let $S\in\R_+^{m\times n}$ and let $S_i\in\R_+^{m_i\times n_i}$ for $i \in [2]$ such that $S=S_1\otimes S_2$. Matrix $S$ is the slack matrix of a polytope $P$ if and only if there exist polytopes $P_i$, $i \in [2]$ such that $S_i$ is the slack matrix of $P_i$ and $P$ is affinely equivalent to $P_1 \times P_2$.
\end{lem}
 \begin{proof}
For $i \in [2]$, let $C_i := \col(S_i)$. From Theorem \ref{thm:char_slack-matrices}, and since $\col(S_1 \otimes S_2) = \col(S_1) \times \col(S_2) = C_1 \times C_2$, it suffices to prove (i) 
\(
 \aff(C_1 \times C_2) = \aff(C_1) \times \aff(C_2)
\)
and (ii)
\(
\conv(C_1 \times C_2) = \conv(C_1) \times \conv(C_2).
\)

We first prove (i). 
We have $\aff(C_1 \times C_2) \subseteq \aff(C_1) \times \aff(C_2)$ since the right-hand side is an affine subspace containing $C_1 \times C_2$. 
Now, we prove $\aff(C_1) \times \aff(C_2) \subseteq \aff(C_1 \times C_2)$. Take $p = (x,y) \in \aff(C_1) \times \aff(C_2)$. Then $x$ is an affine combination $\sum_i \lambda_i x_i = x$ of points $x_i \in C_1$. Similarly, $y$ is an affine combination $\sum_j \mu_j y_j = y$ of points $y_j \in C_2$. Thus we can write $p$ as
$$
(x,y) = \Big(\sum_i \lambda_i x_i, \sum_j \mu_j y_j\Big) 
= \Big(\sum_i \lambda_i \overbrace{(\sum_j \mu_j)}^{=1} x_i, \sum_j \overbrace{(\sum_i \lambda_i)}^{=1} \mu_j y_j\Big)
= \sum_{i,j} \lambda_i\mu_j (x_i,y_j),
$$
where $\sum_{i,j} \lambda_i\mu_j = \left(\sum_i \lambda_i\right) \left(\sum_j\mu_j\right) = 1$. Hence, $p \in \aff(C_1 \times C_2)$.
Moreover, if $\mu_j\geq 0, \lambda_i\geq 0$ for all $i,j$, then the multipliers are all non-negative, which proves (ii).
\end{proof}

\subsection{2-\op{}s}\label{sec:2prod}

We now define the operation of 2-\op, and show that, under certain natural assumptions, it also preserves the property of being a slack matrix.

Consider two real matrices $S_1,S_2$, and assume that $S_1$ (resp.~$S_2$) has a 0/1 row $x_1$ (resp.~$y_1$), that is, a row whose entries are 0 or 1 only. We call $x_1, y_1$ \emph{special} rows. For any matrix $M$ and row $r$ of $M$, we denote by $M-r$ the matrix obtained from $M$ by removing row $r$. The row $x_1$ determines a partition of $S_1-x_1$ into two submatrices according to its 0 and 1 entries: we define $S_1^0$ to be the matrix obtained from $S_1$ by deleting the row $x_1$ and all the columns whose $x_1$-entry is 1, and $S_1^1$ is defined analogously. Thus,
 
\begin{center}
\begin{tikzpicture}[every left delimiter/.style={xshift=.5em},
    every right delimiter/.style={xshift=-.5em}]
\matrix(S)[matrix of math nodes,left delimiter={(},right
delimiter={)},column sep=-\pgflinewidth,row sep=-\pgflinewidth,nodes={anchor=center,text width=1.2cm,align=center,inner sep=0pt,text depth=0},
row 1/.style={nodes={minimum height=4mm}},row 2/.style={minimum height=6mm}]
{
S_1^0 & S_1^1 \\
 0 \cdots 0 & 1 \cdots 1 \\
};
\draw(S-1-1.north east) -- (S-2-1.south east);
\node [left= of S.center, xshift=-5mm] {$S_1 = $};
\node [right,xshift=1.5mm,yshift=-.5mm] at (S-2-2.east) {$\leftarrow x_1\,.$};
\end{tikzpicture}
\end{center}

Similarly, $y_1$ induces a partition of $S_2-y_1$ into $S_2^0,S_2^1$. Here we assume that none of $S_1^0, S_1^1, S_2^0,S_2^1$ is empty, that is, we assume that the special rows contain both 0's and 1's.


The \emph{$2$-\op} of $S_1 \in \R^{m_1 \times n_1}$ with special row $x_1$ and $S_2 \in \R^{m_2 \times n_2}$ with special row $y_1$ is defined as:
$$
S = (S_1,x_1) \otimes_2 (S_2,y_1) := \left(\begin{array}{c|c}
S_1^0 \otimes S_2^0 & S_1^1 \otimes S_2^1 \\
 0 \cdots 0 & 1 \cdots 1
\end{array}
\right)\,.
$$ 

Similarly as before, we say that $S$ is a 2-\op{} if there exist matrices $S_1, S_2$ and 0/1 rows $x_1$ of $S_1$, $y_1$ of $S_2$, such that $S$ is isomorphic to $(S_1,x_1)\otimes_2 (S_2,y_1)$. Again, we will abuse notation and write $S=(S_1,x_1)\otimes_2 (S_2,y_1)$.

For a polytope $P$ with slack matrix $S$, consider a row $r$ of $S$ corresponding to an inequality $a^\intercal x\leq b$ that is valid for $P$. We say that $r$ is \emph{2-level} with respect to $S$, and that 
$a^\intercal x \leq b$ is \emph{2-level} with respect to $P$, if there exists a real $b' < b$ such that all the vertices of $P$ either lie on the hyperplane $\{x \mid a^\intercal x = b\}$ or the hyperplane $\{x \mid a^\intercal x = b'\}$. 

We notice that, if $r$ is 2-level, then $r$ can be assumed to be 0/1 after scaling. Moreover, adding to $S$ the row $\mathbf{1} - r$ (that is, the complement of 0/1 row $r$) gives another slack matrix of $P$. Indeed, such row corresponds to the valid inequality $a^\intercal x \geq b'$. 

The latter observation is crucial for our next lemma: we show that, if the special rows are chosen to be 2-level, the operation of 2-\op{} essentially preserves the property of being a slack matrix. We remark that having a 2-level row is a quite natural condition. For instance, for 0/1 polytopes, any non-negativity constraint yields a 2-level row in the corresponding slack matrix. By definition, all facet-defining inequalities of a 2-level polytope are 2-level.
Finally, we would like to mention that the following result could be derived from results from \cite{margot1995composition} (see also \cite{conforti2016projected}), but we give here a new, direct proof.

\begin{lem}\label{lem:2-sum_slack}
Let $S\in\R_+^{m\times n}$ and let $S_i\in\R_+^{m_i\times n_i}$ for $i \in [2]$ such that $S=(S_1,x_1)\otimes_2 (S_2,y_1)$ for some 2-level rows $x_1$ of $S_1$, $y_1$ of $S_2$. The following hold:
\begin{enumerate}[(i)]
\item If both $S_1$ and $S_2$ are slack matrices, then $S$ is a slack matrix.
\item If $S$ is a slack matrix, let $S'_1 := S_1 +(\mathbf{1}-x_1)$ (that is, $S_1$ with the additional row $\mathbf{1}-x_1$), and similarly let $S'_2 := S_2 +(\mathbf{1}-y_2)$. Then both $S_1'$ and $S_2'$ are slack matrices.
\end{enumerate}
\end{lem}

\begin{proof}
(i) Let $P_i := \conv(\col(S_i)) \subseteq \R^{m_i}$ for $i \in [2]$. Recall that $S_i$ is the slack matrix of $P_i$, by Theorem~\ref{thm:char_slack-matrices}. Without loss of generality, $x_1$ and $y_1$ can be assumed to be the first rows of $S_1, S_2$ respectively. We overload notation and denote by $x_1$ the first coordinate of $x$ as a point in $\R^{m_1}$, and similarly for $y \in \R^{m_2}$. Let $H$ denote the hyperplane of $\R^{m_1+m_2}$ defined by the equation $x_1 = y_1$.

We claim that $S$ is a slack matrix of the polytope $(P_1 \times P_2) \cap H$. By Lemma \ref{lem:1-sum_slack}, $S$ is a submatrix of the slack matrix of $(P_1 \times P_2) \cap H$. But the latter might have some extra columns: hence we only need to show that intersecting $P_1 \times P_2$ with $H$ does not create any new vertex.

To this end we notice that no new vertex is created if and only if there is no edge $e$ of $P_1 \times P_2$ such that $H = \{(x,y) \mid x_1 = y_1\}$ intersects $e$ in its interior. Let $e$ be an edge of $P_1 \times P_2$, and let $(v_1,v_2)$ and $(w_1,w_2)$ denote its endpoints, where $v_1, w_1 \in \col(S_1)$ and $v_2,w_2 \in \col(S_2)$. 
By a well-known property of the Cartesian product, $v_1 = w_1$ or $v_2 = w_2$. Suppose that $(v_1,v_2)$ does not lie on $H$. By symmetry, we may assume that $v_{11} < v_{21}$. This implies $v_{11} = 0$ and $v_{21} = 1$, which in turn implies $w_{11} \leq w_{21}$ (since $v_1 = w_1$ or $v_2 = w_2$). Thus $(w_1,w_2)$ lies on the same side of $H$ as $(v_1,v_2)$, and $H$ cannot intersect $e$ in its interior. Therefore, the claim holds and $S$ is slack matrix.\medskip

\noindent (ii) Assume that $S = (S_1,x_1)\otimes_2 (S_2,y_1)$ is a slack matrix. We show that $S_1'=S_1+(\mathbf{1}-x_1)$ is a slack matrix, using Theorem~\ref{thm:char_slack-matrices}. The argument for $S_2'$ is symmetric. It suffices to show that $\aff(\col(S)) \cap \R_+^{m_1} \subseteq \conv(\col(S))$ since the reverse inclusion is obvious. 

Let $x^* \in \aff(\col(S_1')) \cap \R_+^{m_1}$. One has $x^*=\sum_{i\in I} \lambda_i v_i$ for some coefficients $\lambda_i \in \R$ with $\sum_{i\in I} \lambda_i=1$, where $v_i \in \col(S_1')$ for $i\in I$. We partition the index set $I$ into $I_0$ and $I_1$, so that $i\in I_0$ (resp.~$i\in I_1$) if $v_i$ has its $x_1$ entry equal to 0 (resp.~1). 
For simplicity, we may assume that $x_1$ is the first row of $S_1'$, and $\mathbf{1}-x_1$ the second. Then, the first coordinate of $x^*$ is $x^*_1=\sum_{i \in I_1} \lambda_i\geq 0$, and the second is $x^*_2=\sum_{i \in I_0} \lambda_i\geq 0$. Notice that $x_1^* + x_2^* = \sum_{i \in I} \lambda_i = 1$. 

Now, we extend $x^*$ to a point $\tilde{x}\in \aff(\col(S))$ by mapping each $v_i$, $i\in I$ to a column of $S$, as follows. For each $a\in\{0,1\}$, fix an arbitrary column $c_a$ of $S_2^a$, then map each $v_i$ with $i\in I_a$ to the column of $S$ consisting of $v_i$, without its second component, followed by $c_a$. We denote such column by $u_i$, for $i\in I$, and let $\tilde{x} := \sum_{i\in I} \lambda_i u_i$. 

We claim that $\tilde{x} \in\R^m_+$. This is trivial for any component corresponding to a row of $S_1$, since those are components of $x^*$ as well. Consider a component $\tilde{x}_j$ corresponding to a row of $S_2$, and denote by $c_{a,j}$ the corresponding component of $c_a$, for $a=0,1$. We have:
$$
\tilde{x}_j=\sum_{i\in I_0}\lambda_i c_{0,j}+ \sum_{i\in I_1}\lambda_i c_{1,j}=x^*_2 c_{0,j}+x^*_1 c_{1,j} \geq 0.
$$
Now, Theorem \ref{thm:char_slack-matrices} applied to $S$ implies that $\tilde{x} \in\conv(\col(S))$. That is, we can write $\tilde{x}=\sum_{i\in I'} \mu_i u'_i$ where $u'_i \in \col(S)$ and $\mu_i \in \R_+$ for $i \in I'$ and $\sum_{i\in I'}\mu_i=1$. For each $i \in I'$, let $v'_i \in \col(S'_1)$ denote the column vector obtained from $u'_i$ by restricting to the rows of $S_1$ and inserting as a second component $1 - u'_{i,1} = 1 - v'_{i,1}$. We claim that $x^*=\sum_{i\in I'} \mu_i v'_i$, which implies that $x^*\in  \conv(\col(S_1'))$ and concludes the proof. The claim is trivially true for all components of $x^*$ except for the second, for which one has $x^*_2=1-x^*_1$ since $v'_{i,2} = 1-v'_{i,1}$ for all $i \in I'$ by definition of $v'_i$.
\end{proof}

\section{Algorithms}\label{sec:algorithms}

In this section we study the problem of recognizing $1$-\op{}s. Given a matrix $S$, we want to determine whether $S$ is a $1$-\op{}, and find matrices $S_1, S_2$ such that $S=S_1\otimes S_2$. Since we allow the rows and columns of $S$ to be permuted in an arbitrary way, the problem is non-trivial. 

At the end of the section, we extend our methods to the problem of recognizing $2$-\op{}s. We remark that the results in this section naturally extend to a more general operation, the $k$-\op{}, for every constant $k$ (see~\cite{aprile2018thesis} for more details).

We begin with a preliminary observation, which is the starting point of our approach. Suppose that a matrix $S$ is a 1-\op{} $S_1\otimes S_2$. Then the rows of $S$ can be partitioned into two sets $R_1, R_2$, corresponding to the rows of $S_1, S_2$ respectively. We write that $S$ is a 1-\op{} \emph{with respect to} the partition $R_1,R_2$. A column of the form $(a_1,a_2)$, where $a_i$ is a column vector with components indexed by $R_i$ ($i \in [2]$), is a column of $S$ if and only if $a_i$ is a column of $S_i$ for each $i \in [2]$. Moreover, the number of occurrences of $(a_1,a_2)$ in $S$ is just the product of the number of occurrences of $a_i$ in $S_i$ for $i \in [2]$. Under uniform probability distributions on the columns of $S$, $S_1$ and $S_2$, the probability of picking $(a_1,a_2)$ in $S$ is the product of the probability of picking $a_1$ in $S_1$ and that of picking $a_2$ in $S_2$. We will exploit this intuition below.

\subsection{Recognizing 1-\op{}s via submodular minimization} 

First, we recall some notions from information theory, see~\cite{cover2012elements} for a more complete exposition. 
%
Let $A$ and $B$ be two discrete random variables with ranges $\mathcal{A}$ and $\mathcal{B}$ respectively. The \emph{mutual information} of $A$ and $B$ is:
$$
I(A;B) = \sum_{a\in \mathcal{A}, b\in \mathcal{B}} \pr(A=a,B=b) \cdot \log_2 \frac{\pr(A=a,B=b)}{\pr(A=a)\cdot \pr(B=b)}. 
$$
The mutual information of two random variables measures how close is their joint distribution to the product of the two corresponding marginal distributions. 
%
%

We will use the following facts, whose proof can be found in \cite{krause2005near,cover2012elements}. Let $C_1, \ldots, C_m$ be discrete random variables. For $X \subseteq [m]$ we consider the random vectors $C_{X} := (C_i)_{i \in X}$ and $C_{\overline{X}} := (C_i)_{i \in \overline{X}}$, where $\overline{X} := [m] \setminus X$. The function $f : 2^{[m]} \to \R$ such that
\begin{equation}
\label{eq:def_f}
f(X) := I(C_X;C_{\overline{X}})
\end{equation}
will play a crucial role.

\begin{prop}\label{prop:mutual_info}
\begin{enumerate}[(i)]
\item\label{prop:indep} For all discrete random variables $A$ and $B$, we have $I(A;B)\geq 0$, with equality if and only if $A$ and $B$ are independent.
\item\label{prop:submodular}
If $C_1, \ldots, C_m$ are discrete random variables, then the function $f$ as in~\eqref{eq:def_f} is submodular.
\end{enumerate}
\end{prop}

Let $S$ be an $m \times n$ matrix. Let $C := (C_1,\ldots,C_m)$ be a uniformly chosen random column of $S$. That is, $\pr(C = c) = \mu(c)/n$, where $\mu(c)$ denotes the number of occurrences in $S$ of the column $c \in \col(S)$. 

Let $f: 2^{[m]} \rightarrow \R$ be defined as in~\eqref{eq:def_f}. We remark that the definition of $f$ depends on $S$, which we consider fixed throughout the section. The set function $f$ is non-negative (by Proposition~\ref{prop:mutual_info}.\ref{prop:indep}), symmetric (that is, $f(X)=f(\overline{X})$) and submodular (by Proposition~\ref{prop:mutual_info}.\ref{prop:submodular}). 

The next lemma shows that we can determine whether $S$ is a 1-\op{} by minimizing $f$.

\begin{lem}\label{lem:mutual_info}
Let $S\in \R^{m\times n}$, and $\emptyset\neq X\subsetneq [m]$. Then $S$ is a $1$-\op{} with respect to $X,\overline{X}$ if and only if $C_X$ and $C_{\overline{X}}$ are independent random variables, or equivalently (by Proposition \ref{prop:mutual_info}.\ref{prop:indep}), $f(X) = 0$.
\end{lem}
\begin{proof}
First, we prove ``$\Longrightarrow$''. Suppose that $S$ is a $1$-\op{} with respect to $X,\overline{X}$ for some non-empty and proper set $X$ of 
row indices of $S$. Let $S=S_1\otimes S_2$ be the corresponding $1$-\op{}, where $S_i\in \R^{m_i\times n_i}$ for $i=1,2$. 

For any column $c=(c_X,c_{\overline{X}}) \in \col(S)$, we have $\mu(c)=\mu_1(c_X)\mu_2(c_{\overline{X}})$, where $\mu_i$ denotes the multiplicity of a column in $S_i$, $i=1,2$. Hence
\begin{align*}
\pr(C_X=c_X,\ C_{\overline{X}}=c_{\overline{X}})
&= \pr(C=c) = \frac{\mu(c)}{n}\\
= \frac{\mu_1(c_X) \cdot n_2}{n} \cdot\frac{n_1 \cdot \mu_2(c_{\overline{X}})}{n}
&=\pr(C_X=c_X)\cdot\pr(C_{\overline{X}}=c_{\overline{X}}),
\end{align*}
where we used $n=n_1n_2$. This proves that $C_X$ and $C_{\overline{X}}$ are independent.
 
We now prove ``$\Longleftarrow$''. Let $a_1, \dots,a_k$ denote the different columns of the restriction $S\restr{X}$ of matrix $S$ to the rows in $X$, and $b_1,\dots,b_\ell$ denote the different columns of $S\restr{\overline{X}}$ (the restriction of matrix $S$ to the rows in $\overline{X}$). Since $C_X$ and $C_{\overline{X}}$ are independent, we have that, for any column $c=(a_i, b_j)$ of $S$, 
$$
\mu(a_i,b_j)=n\cdot \pr(C_X=a_i, C_{\overline{X}}=b_j)=n\cdot\pr(C_X=a_i) \pr(C_{\overline{X}}=b_j)=\frac{\mu_X(a_i)\mu_{\overline{X}}(b_j)}{n},
$$
where $\mu_X(\cdot)$ and $\mu_{\overline{X}}(\cdot)$ denote multiplicities in $S\restr{X}$ and $S\restr{\overline{X}}$ respectively. 

Now, let $M$ denote the $k \times \ell$ matrix such that $M_{i,j}:=\mu(a_i, b_j)$. We have shown that $M$ is a non-negative integer matrix with a rank-$1$ non-negative factorization of the form $uv^\intercal$, where $u_i := \mu_X(a_i)/n$ and $v_j := \mu_{\overline{X}}(b_j)$, for $i \in [k]$ and $j \in [\ell]$. 

Next, one can easily turn this non-negative factorization into an integer one. Suppose that $u_i$ is fractional for some $i \in [k]$. Writing $u_i$ as $u_i = p_i/q_i$, where $p_i \in \Z_{\geq 0}$ and $q_i \in \Z_{>0}$ are coprime, we see that $q_i$ divides $v_j$ since $u_iv_j$ is integer, for every $j \in [\ell]$. Then the factorization $q_i u \cdot \frac{1}{q_i} v^\intercal = u'(v')^\intercal$ is such that $v'$ is integer and $u'$ has at least one more integer component than $u$. Iterating this argument, we obtain that $M=\overline{u}\,\overline{v}^\intercal$ where $\overline{u},\overline{v}$ have non-negative integer entries. 

Finally, let $S_1$ be the matrix consisting of the column $a_i$ repeated $\overline{u}_i$ times for $i \in [k]$, and construct $S_2$ from $\overline{v}$ in an analogous way. Then it is immediate to see that $S=S_1\otimes S_2$ and in particular $S$ is a 1-\op{} with respect to the row partition $X,\overline{X}$, which concludes the proof.
\end{proof}

Notice that the previous proof also gives a way to efficiently reconstruct $S_1$, $S_2$ once we identified $X$ such that $f(X)=0$.  In particular, if the columns of $S$ are all distinct, then reconstructing $S_1, S_2$ is immediate: $S_1$ consists of all the distinct columns of $S\restr{X}$, each taken once, and $S_2$ is obtained analogously from $S\restr{\overline{X}}$. The last ingredient we need is that every (symmetric) submodular function can be minimized in polynomial time. Here we assume that we are given a polynomial time oracle to compute our function.

\begin{thm}[Queyranne~\cite{queyranne1998minimizing}] \label{thm:submodular}
There is a polynomial time algorithm that outputs a set $X$ such that $X \neq \emptyset, A $ and $f(X)$ is minimum, where $f : 2^{[m]}\rightarrow \R$ is any given symmetric submodular function.
\end{thm}

As a direct consequence, we obtain Theorem \ref{thm:main1}. 

\begin{proof}[Proof of Theorem \ref{thm:main1}] It is clear that $f(X)$ can be computed in polynomial time for any $X$. It suffices then to run Queyranne's algorithm to find $X$ minimizing $f$. If $f(X)>0$, then $S$ is not a 1-\op. Otherwise, $f(X)=0$ and $S_1$, $S_2$ can be reconstructed as described in the proof of Lemma~\ref{lem:mutual_info}.
\end{proof}

We conclude the section with a decomposition result which will be useful in the next section. We call a matrix \emph{irreducible} if it is not a 1-\op. The result below generalizes the fact that a polytope can be uniquely decomposed as a cartesian product of ``irreducible'' polytopes. 


\begin{lem}\label{lem:unique_decomposition}
Let $S\in \R^{m\times n}$ be a 1-\op. Then there exists a partition $\{X_1,\dots, X_t\}$ of $[m]$ such that:
\begin{enumerate}
\item $S$ is a $1$-\op{} with respect to $X_i, \overline{X_i}$ for all $i \in [t]$;
\item for all $i \in [t]$ and all proper subsets $X$ of $X_i$, $S$ is not a 1-\op{} with respect to $X,\overline{X}$;
\item the partition $X_1, \dots, X_t$ is unique up to permuting the labels.
\end{enumerate}
In particular, if $S$ has all distinct columns, then there are matrices $S_1,\dots, S_t$ such that $S=S_1\otimes\dots\otimes S_t$, each $S_i$ is irreducible, and the choice of the $S_i$'s is unique up to renaming and permuting columns.
\end{lem}
\begin{proof}
Let $f:2^{[m]} \to \R$ be the function defined in Equation~\ref{eq:def_f}. Let $\mathcal M =\{X \subseteq [m] \mid f(X)=0\}$.  Let $X_1, \dots, X_t$ be the minimal (under inclusion) non-empty members of $\mathcal M$.  Since $f$ is non-negative and submodular, if $f(A)=f(B)=0$, then $f(A\cap B)=f(A\cup B)=0$. By minimality, this implies that $X_i \cap X_j = \emptyset$ for all $i \neq j$.  Since $f$ is symmetric, $\bigcup_{i \in [t]} X_i=[m]$. By Lemma~\ref{lem:mutual_info},  $t \geq 2$ and $X_1, \dots, X_t$ satisfy (i) and (ii).  Conversely, by Lemma~\ref{lem:mutual_info}, if $\{Y_1, \dots, Y_s\}$ is a partition of $[m]$ satisfying (i) and (ii), then $Y_1, \dots, Y_s$ are the minimal non-empty members of $\mathcal M$, which proves uniqueness.  



To conclude, assume that $S$ has all distinct columns. Then as argued above each $S_i$ is obtained by picking each distinct column of $S\restr{X_i}$ exactly once, and it is thus unique up to permutations, once $X_i$ is fixed. Each $S_i$ is irreducible thanks to the minimality of $X_i$ and to Lemma \ref{lem:mutual_info}. The fact that the $X_i$'s are unique up to renaming concludes the proof.
\end{proof}

\subsection{Extension to $2$-\op{}s}
We now extend the previous results to obtain a polynomial algorithm to recognize $2$-\op{}s. Recall that, if a matrix $S$ is a $2$-\op, then it has a special row that divides $S$ in submatrices $S^{0}, S^{1}$, which are 1-\op{}s \emph{with respect to the same partition}. Hence, our algorithm starts by guessing the special row, and obtaining the corresponding submatrices $S^{0}, S^{1}$. Let $f_0$ (resp. $f_1$) denote the function $f$ as defined in \eqref{eq:def_f} with respect to the matrix $S^{0}$ ($S^{1}$), and let $\tilde{f}=f_0+f_1$. Notice that $\tilde{f}$ is submodular, and is zero if and only if each $f_i$ is. Let $X$ be a proper subset of the non-special rows of $S$ (which are the rows of $S^{0}$ and $S^{1}$). It is an easy consequence of Lemma \ref{lem:mutual_info} that $S^{0}, S^{1}$ are 1-\op{}s with respect to $X$ if and only if $\tilde{f}(X)=0$. Then $S$ is a $2$-\op{} with respect to the chosen special rows if and only if the minimum of $\tilde{f}$ is zero.

Once a feasible partition is found, $S_1, S_2$ can be reconstructed by first reconstructing all $S_1^0, S_1^1, S_2^0, S_2^1$ and then concatenating them and adding the special rows.
We obtained the following:
\begin{thm}
\label{thm:2-sum_recog}
Let $S\in \R^{m\times n}$. There is an algorithm that is polynomial in $m,n$ and determines whether $S$ is a $2$-\op{} and, in case it is, outputs two matrices $S_1, S_2$ and special rows $x_1$ of $S_1$, $y_1$ of $S_2$, such that $S=(S_1,x_1)\otimes_2 (S_2, y_1)$.
\end{thm}

In order to apply Theorem \ref{thm:2-sum_recog} to decompose slack matrices, we need to deal with a last issue. In the algorithm, it is essential to guess the special row that partitions the column set in 1-\op{}s. However, in principle there might be a slack matrix that is obtained as $2$-\op{} of other slack matrices, but where the special row is redundant. Then, deleting such row still gives a slack matrix, but we cannot recognize such matrix as $2$-\op{} any more using our algorithm. However, the next lemma ensures that this does not happen, as long as the special rows are not redundant in the factors of the $2$-\op.

\begin{lem}\label{lem:k-sum_non-redundant}
Let $S\in\R^{m\times n}$ and let $S_i\in\R^{m_i\times n_i}$ for $i=1,2$ such that $S=(S_1,x_1) \otimes_2 (S_2,y_1)$ for some special rows $x_1$ of $S_1$ and $y_1$ of $S_2$. Assume that $S_1, S_2, S$ are slack matrices, and that the rows $x_1$, $y_1$ are non-redundant for $S_1, S_2$ respectively. Then the special row $r$ in $S$ is non-redundant as well.
\end{lem}
\begin{proof}
Notice that, in any slack matrix, a row $r$ is redundant if and only if its set of zeros is strictly contained in the set of zeros of another row $r'$ (we write that $r'$ \emph{dominates} $r$ for brevity).
Assume by contradiction that $r$ is redundant, and let $r'$ be another row of $S$ such that $r'$ dominates $r$. Let us assume by symmetry that $r'$ corresponds to a row $r_1$ of $S_1$, i.e. $r'$ consists (up to permutation) of $r_1$ repeated $n_2$ times. Then it is clear that $r_1$ dominates $x_1$, hence $x_1$ is redundant in $S_1$.
\end{proof} 

\section{Application to 2-level matroid base polytopes}\label{sec:slackmatroids}

In this section, we use the results in Section \ref{sec:algorithms} to derive a polynomial time algorithm to recognize the slack matrix of a 2-level base matroid polytope. 

We start with some basic definitions and facts about matroids, and we refer the reader to \cite{oxley2006matroid} for missing definitions and details. We regard a matroid $M$ as a couple $(E,\calB)$, where $E$ is the ground set of $M$, and $\calB$ is its set of bases. The dual matroid of $M$, denoted by $M^*$, is the matroid on the same ground set whose bases are the complements of the bases of $M$. 
An element $p\in E$ is called a \emph{loop} (respectively \emph{coloop}) of $M$ if it appears in none (all) of the bases of $M$. 
Given an element $e\in E$, the \emph{deletion} of $e$ is the matroid $M-e$ on $E\setminus \{e\}$ whose bases are the bases of $M$ that do not contain $e$. The \emph{contraction} of $e$ is the matroid $M/e$ on $E\setminus \{e\}$ whose bases are of the form $B\setminus \{e\}$, where $B$ is a basis of $M$ that contains $e$. 
A matroid $M = (E,\calB)$ is \emph{uniform} if $\calB=\binom{E}{k}$, where $k$ is the rank of $M$. We denote the uniform matroid with $n$ elements and rank $k$ by $U_{n,k}$.

Consider matroids $M_1=(E_1,\calB_1)$ and $M_2=(E_2,\calB_2)$, with non-empty ground sets. If $E_1 \cap E_2 =\emptyset$, the \emph{1-sum} $M_1 \oplus M_2$ is defined as the matroid with ground set $E_1 \cup E_2$ and base set $\calB_1 \times \calB_2$. If, instead, $E_1 \cap E_2 = \{p\}$, where $p$ is neither a loop nor a coloop in $M_1$ or $M_2$, 
we let the \emph{2-sum} $M_1 \oplus_2 M_2$ be the matroid with ground set $(E_1 \cup E_2) \setminus \{p\}$, and base set $\{(B_1 \cup B_2) \setminus \{p\} \mid B_i\in \calB_i \mbox{ for } i=1,2 \mbox{ and } p\in B_1 \triangle B_2\}$.  A matroid is \emph{connected} if it cannot be written as the 1-sum of two matroids, each with fewer elements. It is well known that $M_1 \oplus_2 M_2$ is connected if and only if so are $M_1$ and $M_2$.

The \emph{base polytope} $B(M)$ of a matroid $M$ is the convex hull of the characteristic vectors of its bases. It is well known that:
\[
B(M)=\{x\in\R^E_+: x(U)\leq \rk(U) \ \forall U\subseteq E, x(E)=\rk(E)\},
\]
where $\rk$ denotes the rank function of $M$.
It is easy to see that, if $M=M_1\oplus M_2$, then $B(M)$ is the cartesian product $B(M_1)\times B(M_2)$, hence its slack matrix is a 1-\op{} thanks to Lemma \ref{lem:1-sum_slack}. If $M=M_1\oplus_2 M_2$, then a slightly less trivial polyhedral relation holds, providing a connection with the 2-\op{} of slack matrices. We will explain this connection below.
We remark that, for any matroid $M$, the base polytopes $B(M)$ and $B(M^*)$ are affinely equivalent via the transformation $f(x)=1-x$ and hence have the same slack matrix. 

Our algorithm is based on the following decomposition result, that characterizes those matroids $M$ such that $B(M)$ is 2-level (equivalently, such that $B(M)$ admits a 0/1 slack matrix). 

\begin{thm}[\cite{Grande16}] \label{thm:matroid-2-level} 
	The base polytope of a matroid $M$ is $2$-level if and only if $M$ can be obtained from uniform matroids through a sequence of 1-sums and 2-sums.
\end{thm}


The general idea is to use the algorithms from Theorems \ref{thm:main1}, \ref{thm:2-sum_recog} to decompose our candidate slack matrix as 1-\op{} and 2-\op{}, until each factor corresponds to the slack matrix of a uniform matroid. The latter can be easily recognized. Indeed, the base polytope of the uniform matroid $U_{n,k}$ is the $(n,k)$-hypersimplex
\(
B(U_{n,k}) = \{ x \in [0,1]^E \mid \sum_e x_e = k \}
\).
If $2\leq k \leq n-2$, the (irredundant, 0/1) slack matrix $S$ of $B(U_{n,k})$ has $2n=2|E|$ rows and $\binom{n}{k}$ columns of the form $(v, \mathbf{1}-v)$ where $v\in\{0,1\}^n$ is a vector with exactly $k$ ones, hence can be recognized in polynomial time (in its size). We denote such matrix by $S_{n,k}$. If $k=1$, or equivalently $k=n-1$, $S=S_{n,1}=S_{n,n-1}$ is just the identity matrix $I_n$. The case $k=0$ or $k=n$ corresponds to a non-connected matroid whose base polytope is just a single vertex, and can be ignored for our purposes.

Before going further, we need some preliminary assumptions. Let $M(E,\calB)$ be a matroid such that $B(M)$ is 2-level, and let $S$ be a 0/1 slack matrix of $B(M)$. From now on we assume that:
\begin{enumerate}
\item \label{assump:loops} $M$ does not have loops or coloops. 
\item \label{assump:const} $S$ does not have any constant row (i.e. all zeros or all ones).
\item \label{assump:nonneg}$S$ has a row for each inequality of the form $x(e)\geq 0$ for $e\in E$ (we refer to such rows as non-negativity rows).

\end{enumerate} 
Assumption \ref{assump:loops} is without loss of generality as, if $e$ is a loop or coloop of $M$, then $B(M)$ has a constant coordinate in correspondence of $e$ and is thus isomorphic to $B(M-e)$. Similarly, Assumption \ref{assump:const} is without loss of generality as constant rows correspond to redundant inequalities and can always be removed from a slack matrix.

We now justify Assumption \ref{assump:nonneg}. 
One can show (directly, or using well known facts from~\cite{tutte1965lectures} and~\cite{feichtner2005matroid}) that for each element $e\in E$ at least one of the inequalities $x_e\geq 0$, $x_e\leq 1$ is facet defining for $B(M)$. Notice that these form pairs of opposite 2-level inequalities, and recall from the discussion in Section \ref{sec:2prod} that for any slack matrix $S$ with a 2-level row $r$, adding the opposite row (i.e. $\mathbf{1}-r$) does not change the fact that $S$ is a slack matrix (or a 1-\op{} or a 2-\op{}). Hence we can assume that our slack matrix $S$ contains all the non-negativity rows.

We now focus on the relationship between 1-sums and 1-\op{}s. As already remarked, if $S_1, S_2$ are the slack matrices of $B(M_1), B(M_2)$ respectively, then $S_1\otimes S_2$ is the slack matrix of $B(M_1)\times B(M_2)=B(M_1\oplus M_2)$. We now show that the converse holds, i.e. we need to make sure that, whenever we decompose the slack matrix of a matroid base polytope as a 1-\op, the factors are still matroid base polytopes.

\begin{lem} \label{lem:1-\op_matroid}
Let $M$ be a matroid and let $S$ be the slack matrix of $B(M)$. If $S=S_1\otimes S_2$ for some matrices $S_1, S_2$, then there are matroids $M_1, M_2$ such that $M=M_1\oplus M_2$ and $S_i$ is the slack matrix of $B(M_i)$ for $i=1,2$.
\end{lem}
 \begin{proof} 
By Assumption \ref{assump:nonneg}, $S$ contains all the rows corresponding to inequalities $x(e)\geq 0$, for any $e$ element of $M$. Each such non-negativity inequality belongs either to $S_1$ or to $S_2$, hence we can partition $E$ into $E_1, E_2$ accordingly. Recall that the row set of $S$ can also be partitioned into sets $R_1, R_2$, as each row of $S$ corresponds to a row of $S_1$ or $S_2$. Notice that none of $E_1, E_2$ can be empty: if for instance $E_2$ is empty, then all the rows corresponding to $x(e)\geq 0$ belong to $R_1$. But then the slack of a vertex with respect to every other inequality (of form $x(U)\leq \rk(U)$ for some $U\subseteq E$) depends entirely on the slack with respect to the rows in $R_1$, implying that a column of $S\restr{R_1}$ can be completed to a column of $S$ in a unique way. Hence, since $S$ is a 1-\op, we must conclude that $S_2$ is made of a single column, contradicting the fact that $S$ does not have constant rows (Assumption \ref{assump:const}). 

Now, let $\calB_i=\{B\cap E_i: B\in \calB\}$ for $i=1,2$. By definition of 1-\op{} of matrices, $B(M)=\{B_1\cup B_2: B_i\in \calB_i $ for $i=1,2\}$. 
This implies that $M=M_1\oplus M_2$ where $M_i=M\restr{E_i}$  for $i=1,2$, thus $B(M)=B(M_1)\times B(M_2)$. Hence, for every row of $S$ corresponding to an inequality $x(U)\leq \rk(U)$, we have either $U\subseteq E_1$, $U\subseteq E_2$, or the inequality is redundant and can be removed. In the first case, clearly the row is in $R_1$ as its entries depend only on the rows $x(e)\geq 0$ for $e\in E_1$, and similarly in the second case the row is in $R_2$. As by removing redundant rows we do not change the polytopes of which $S,S_1,S_2$ are slack matrices, we then conclude that $S_i$ is a slack matrix of $B(M_i)$ for $i=1,2$.
\end{proof}

\begin{cor}\label{cor:matroid_conn}
Let $M$ be a matroid and let $S$ be the slack matrix of $B(M)$. Then $M$ is connected if and only if $S$ is irreducible. 
\end{cor}

Now, we deal with slack matrices of connected matroids and with the operation of 2-\op. 
We will need the following result, which provides a description of the base polytope of a 2-\op{} $M_1\oplus_2 M_2$ in terms of the base polytopes of $M_1, M_2$. Its proof can be derived from \cite{Grande16}, or found in \cite{aprile20182}.

\begin{lem}\label{obs:2sumpolydescription}
	Let $M_1, M_2$ be matroids on ground sets $E_1, E_2$ respectively, with $E_1\cap E_2=\{p\}$, and let $M=M_1\oplus_2 M_2$. Then $B(M)$ is affinely equivalent to $$(B(M_1)\times B(M_2))\cap\{(x,y) \in\R^{E_1} \times \R^{E_2} \mid x_{p}+y_{p}=1\}.$$
\end{lem}

Lemma \ref{obs:2sumpolydescription} implies that if $M=M_1\oplus_2 M_2$ and $S_i$ is a slack matrix of $B(M_i)$ for $i=1,2$, then the slack matrix of $B(M)$ is actually $(S_1,x_p) \otimes_2 (S_2,\overline{y}_p)$, where $x_p$ is the row corresponding to $x_p\geq 0$, and $\overline{y}_p$ the row corresponding to $y_p\leq 1$. If the special rows $x_p, \overline{y}_p$ have this form, we say that they are \emph{coherent}.

The only missing ingredient is now a converse to the above statement. In particular, we would need that if the slack matrix of a base polytope is a 2-\op, then the corresponding matroid is a 2-sum. We prove this under the additional assumption that one of the factor of the 2-\op{} corresponds to a uniform matroid: this is an assumption we can always make, thanks to Theorem \ref{thm:matroid-2-level}.

\begin{lem} \label{lem:2-\op_matroid}
Let $M = (E,\calB)$ be a connected matroid and let $S$ be the slack matrix of $B(M)$. Assume there are $S_1, S_2$ such that $S=(S_1,x_1)\otimes_2 (S_2,\overline{y}_1)$, for some 2-level rows $x_1, \overline{y}_1$, and let $S_1'=S_1+(\mathbf{1}-x_1)$ and similarly for $S_2'$. Assume that $S_1$ or $S_1'$ is equal to $S_{d,k}$ for some $d>k\geq 1$. Then there is a matroid $M_2$ such that $M=U_{d,k}\oplus_2 M_2$ and $S_2'$ is a slack matrix of $B(M_2)$.
\end{lem}

\begin{proof} 
We first claim that the special row $r$ of $S$ does not correspond to any non-negativity inequality (which are all present in $S$ thanks to Assumption \ref{assump:nonneg}): indeed, if it corresponds to $x(e)\geq 0$ for some $e\in E$, then it is not hard to see that $S^{00}$ is the slack matrix of $M-e$, and similarly $S^{11}$ is the slack matrix of $M/e$. But both matrices are 1-\op{}s, hence by Corollary \ref{cor:matroid_conn}, none of $M-e, M/e$ is connected. But this is in contradiction with the well known fact (see~\cite{tutte1965lectures}) that, if $M$ is connected, then at least one of $M-e, M/e$ is.

Hence, each inequality $x(e)\geq 0$ corresponds to a row of either $S_1$ or $S_2$, giving a partition of $E$ in $E_1, E_2$. We will now proceed similarly as in the proof of Lemma \ref{lem:1-\op_matroid}: first, by noticing that the slack of any vertex with respect to $x(U)\leq \rk(U)$ depends exclusively on the slack with respect to the non-negativity inequalities, we can again conclude that $E_1, E_2$ are not empty.
Since $S_1=S_{n,k}$ is the slack matrix of $U_{n,k}$, the special row $x_1$ of $S_1$ corresponds to the inequality $x(p)\geq 0$, or $x(p)\leq 1$ for some element $p$: we can assume that $S_1$ contains both rows (which are opposite), so that we do not need to mention $S_1'$, and similarly for $S_2$, and we consider the case in which $x_1$ corresponds to $x(p)\geq 0$, the other being analogous. Notice that $p$ is not in $E$, as the special row of $S$ does not correspond to a non-negativity inequality. Let us define $M_1=U_{n,k}$ on ground set $E_1'=E_1+p$, with base set $\calB_1=\binom{E_1'}{n}$ and let:
$$\calB_2=\{B_2+p: B_1\cup B_2\in \calB, B_1\subset E_1, |B_1|=k\}\cup  \{B_2: B_1\cup B_2\in \calB, B_1\subset E_1, |B_1|=k-1\}.
$$
We now claim that $M_2$ with ground set $E_2'=E_2+p$ and base set $\calB_2$ is a matroid. Proving this claim will conclude the proof: notice that due to the 2-\op{} structure of $S$, a basis $B_2 \in \calB_2$ can be completed to a basis of $M$ by adding any $B_1\in \calB_1$ that satisfies $p\in B_1\Delta B_2$, and removing $p$. This implies that $M=M_1\oplus_2 M_2$.
Hence $B(M)$ is isomorphic to $B(M_1)\times B(M_2)\cap \{x_p+y_p=1\}$, thanks to Lemma \ref{obs:2sumpolydescription}, and can be described by: a description of $B(M_1)$, a description of $B(M_2)$, and the equations $x(E_1)+y(E_2)=\rk(M), x_p+y_p=1$, which do not appear in the slack matrix. Now, as $S_1$ is the slack matrix of $B(M_1)$, the rows of $R_2$ must correspond to a description of $B(M_2)$: from this it follows that $S_2$ is a slack matrix of $B(M_2)$, concluding the proof.

To prove the claim, we now show that $\calB_2$ satisfies the axioms for the base set of a matroid: it is non-empty (which is clear) and for any $B_2, B_2'\in \calB_2$ and $e\in B_2\setminus B_2'$, there exists $f\in B_2'\setminus B_2$ such that  $B_2-e+f\in \calB_2$. We fix such $B_2, B_2', e$ and distinguish a number of cases.
\begin{enumerate}
\item $p\in B_2\cap B_2'$. Then for any $B_1\in\calB_1$ with $p\not\in B_1$ we have that $B_1\cup B_2-p, B_1\cup B_2' -p$ are bases of $M$, hence by applying the base axiom to them we obtain that there exists $f\in (B_1\cup B_2'-p)\setminus (B_1\cup B_2 -p)=B_2'\setminus B_2$ such that $B_1\cup B_2-p-e+f \in \calB$, but then $B_2-e+f\in \calB_2$ by definition.
\item $p\not\in B_2\cup B_2'$. This case is analogous to the previous one.
\item $p \in B_2\setminus B_2'$, and $e\neq p$. Let $B_1, B_1'\in \calB$ with $B_1\Delta B_1'=\{p,g\}$ and in particular $p\in B_1'\setminus B_1$. Then we have $B=B_1\cup B_2-p, B'=B_1'\cup B_2'-p \in \calB$.  Then by the base axiom there exists $f \in B'\setminus B$ with $B-e+f\in \calB$. Since $g\in B_1\setminus B_1'\subset B\setminus B'$, we have $g\neq f$ and we conclude that $f\in B_2'\setminus B_2$, hence $B_2-e+f \in \calB_2$.
\item $p \in B_2\setminus B_2'$, and $e=p$. This case is analogous to the previous one, but we apply the axiom to $g\in B\setminus B'$ instead of $e$.
\item $p \in B_2'\setminus B_2$. Let $B_1, B_1'\in \calB$ with $B_1\triangle B_1'=\{p,g\}$ and in particular $p\in B_1\setminus B_1'$, then again  $B=B_1\cup B_2-p, B'=B_1'\cup B_2'-p \in \calB$ and there is $f \in B'\setminus B$ with $B-e+f\in \calB$. If $f \in B_2'$, then $f\not\in B_2$ an we are done as before. Otherwise $f=g$, then $B-e+g=B_1'\cup B_2-e\in \calB$, but then by definition $B_2-e+p \in \calB_2$.  
\end{enumerate}
\end{proof}

We are now ready to prove the main result of this section, namely, Theorem \ref{thm:main3}.


\begin{proof}[Proof of Theorem~\ref{thm:main3}]
We first check whether $S=S_{d,k}$ for some $d$ and $k$, in which case we are done.

Then, we run the algorithm to recognize 1-\op{}s, and if $S$ is a 1-\op, we decompose it in irreducible factors $S_1,\dots, S_t$ and test each $S_i$ separately. This can be done efficiently thanks to Theorem \ref{thm:main1}, and using Lemma \ref{lem:1-\op_matroid} we have that $S$ is the slack matrix of $B(M)$ if and only if $S_i$ is the slack matrix of $B(M_i)$ for each $i$, and $M=M_1\oplus \dots \oplus M_t$.

We can now assume that $S$ is irreducible, and apply the algorithm from Theorem \ref{thm:2-sum_recog} until we decompose $S$ as a repeated 2-\op{} of matrices $S_1,\dots,S_t$ where $S_i=S_{d_i,k_i}$ for $i=1,\dots,t$ (of course, if this is not possible, we conclude that $S$ is not a slack matrix of a base polytope). There is one last technicality we have to deal with, before we can conclude that $S$ is the slack matrix of a matroid polytope. Indeed, as noticed above, we need to ensure that each pair of special rows involved in a 2-\op{} is coherent. 
Note that, unless $S_i$ is the identity matrix (in which case all its rows are non-negativity rows), we can choose whether $S_i$ is the slack matrix of $U_{d_i,k_i}$ or of its dual $U_{d_i,d_i-k_i}$, hence we can choose the form of the special row. Hence $S$ is the slack matrix of a matroid polytope if and only if there is a choice that makes all the pairs of special rows coherent. This problem can be easily solved as follows: define a tree with nodes $S_1,\dots,S_t$, where two nodes $S_i, S_j$ are joint if the 2-\op{} $S_i\otimes_2 S_j$ occurs during the decomposition of $S$. Now, by coloring the nodes of the tree by two colors, according to the form of the special row, one can efficiently determine whether there exists a proper coloring satisfying the ``fixed" colors (given by the $S_i$'s that are identity matrices). Notice that, if there exists a feasible coloring, then this determines a matroid $M$, and it is essentially unique: it is easy to see that the only other possible coloring gives rise to the dual matroid $M^*$, corresponding to the same slack matrix. This concludes the algorithm. Notice that, in case $S$ is the slack matrix of $B(M)$, $M$ (or its dual) can be reconstructed by successively taking the 2-\op{} of the $U_{d_i,k_i}$'s (or of their duals, depending on the coloring found).
\end{proof}



\section*{Acknowledgements}

This project was supported by ERC Consolidator Grant 615640-ForEFront and by a gift by the SNSF.

\printbibliography

\end{document}